\documentclass[reqno]{amsart}
\usepackage{graphicx}
\usepackage{amssymb}
\usepackage{amsfonts}
\usepackage{indentfirst}
\usepackage{amssymb,amsmath,amsthm}
\usepackage{latexsym,bm}
\usepackage{graphicx}
\usepackage{times}
\usepackage{amsfonts}
\usepackage{indentfirst}
\usepackage[colorlinks,linkcolor=black,anchorcolor=blue,citecolor=blue]{hyperref}
\usepackage[dvipsnames]{pstricks}
\newtheorem{theorem}{Theorem}
\newtheorem{proposition}{Proposition}

\newtheorem{lemma}{Lemma}

\numberwithin{equation}{section}

\def\Z{\mathbb{Z}}

\def\Zp{\mathbb{Z}_{p}}

\def\X{\mathbb{X}}

\def\Z2{\mathbb{Z}_{2}}
\def\m2#1{\ ({\rm mod} \ 2^{#1})}

\begin{document}

\title{Dynamics of Chebyshev polynomials on $\mathbb{Z}_{2}$ }

\author{Shilei Fan}

\address{School of Mathematics and Statistics \& Hubei Key Laboratory of Mathematical Sciences, Central  China Normal University,  Wuhan, 430079,  P.  R.  China}
\email{slfan@mail.ccnu.edu.cn}

\author{Lingmin Liao}
\address{LAMA, UMR 8050, CNRS,
Universit\'e Paris-Est Cr\'eteil Val de Marne, 61 Avenue du
G\'en\'eral de Gaulle, 94010 Cr\'eteil Cedex, France}
\email{lingmin.liao@u-pec.fr}

\thanks{Shilei Fan was  supported NSF of China (Grant No.11401236 and 11471132). Lingmin Liao was partially supported by 12R03191A - MUTADIS (France).}
\begin{abstract}
The dynamical structure of  Chebyshev polynomials on $\mathbb{Z}_2$, the ring of $2$-adic integers, 
 is fully described by showing all its minimal subsystems and their attracting basins.

\end{abstract}
\maketitle



\section{Introduction}
For each integer $m\geq 0$,  the $m$-th Chebyshev polynomial is defined as
\begin{align}\label{def-poly}
T_{m}(x)=\sum_{k=0}^{\lfloor m/2 \rfloor}(-1)^{k}\frac{m}{m-k} {m-k \choose k}
2^{m-2k-1}x^{m-2k}.
\end{align}
The Chebyshev polynomials are useful in many parts of analysis, especially in approximation theory. For the fundamental properties and applications of Chebyshev polynomials, we refer to the books \cite{MH03, Riv74}.

Let $p$ be a prime number. Recently, polynomials were studied as $1$-Lipschitz dynamical systems on the ring $\mathbb{Z}_p$ of $p$-adic integers. See  the book \cite{Anashin-Khrennikov-AAD} for the development in this topic. 

A dynamical system is a continuous transformation acting on a topological space. 
To understand a dynamical system, we want to know how a point moves under the iteration of the transformation. A (sub)-system is called {\it minimal} if every orbit is dense in the (sub)-space. The minimality of the polynomial (or $1$-Lipschitz) dynamical systems on $\mathbb{Z}_p$ was widely studied \cite{Ana94,Ana06,AKY11,CP11Ergodic,DP09, FLYZ07, FanSLiao, Jeong2013,OZ75}.

In general, the following theorem shows that a polynomial with degree at least $2$ has only finite number of periodic orbits and has at most countably many minimal subsystems.
\begin{theorem}[\cite{FL11}, Theorem 1]\label{thm-decomposition}
 Let $f \in \mathbb{Z}_p[x]$ be a polynomial of integral coefficients with degree at least  $2$. We have the following decomposition
$$
     \mathbb{Z}_p = \mathcal{P} \bigsqcup \mathcal{M} \bigsqcup \mathcal{B}
$$
where $ \mathcal{P}$ is the finite set consisting of all periodic points of
$f$, $\mathcal{M}= \bigsqcup_i \mathcal{M}_i$ is the union of all (at most countably
many) clopen invariant sets such that each $\mathcal{M}_i$ is a finite union
of balls and each subsystem $f: \mathcal{M}_i \to \mathcal{M}_i$ is minimal, and each
point in $\mathcal{B}$ lies in the attracting basin of a periodic orbit or of
a minimal subsystem.
\end{theorem}

The decomposition in Theorem \ref{thm-decomposition} is usually referred to as a \emph{minimal decomposition} and the invariant subsets $\mathcal{M}_i$ are called {\em minimal components }.

Even though we have a general decomposition theorem, it is difficult to describe  the structure of  a concrete polynomial. The exact minimal decomposition of quadratic  polynomials on $\mathbb{Z}_2$ \cite{FL11} and square mapping on $\Zp$ (for all primes $p$) \cite{FanSLiao} have been investigated. 
In this note, 
we obtain the detailed minimal decomposition of Chebyshev polynomials $T_{m}$ on $\mathbb{Z}_2$.

Since $T_1(x)=x$ is trivial, we only study $T_{m}$ for $m\geq 2$. As will be shown in Theorem \ref{decomposition}, the dynamical structure  of $T_m$ will be  simpler when $m$ is even. Our main task is to study the dynamical system $T_m$ when $m$ is odd.
Note that for any odd number $m\geq 3$, there exists a unique integer $s=s(m)\geq 2$ such that $m=2^sq+1$ or $m=2^sq-1$ with $q\geq 1$ being another odd number. In fact, 
\begin{align}\label{def-s}
s(m):=\max\{n\geq 2: 2^n| (m+1) \ \text{or} \  2^n| (m-1)\}.
\end{align}
 Then we can decompose the set of odd positive integers at least $3$ as
    $$\{m=2k+1: k\geq 1,k\in \mathbb{Z}\}=\bigsqcup_{s\geq2}(\bigsqcup_{q~\text{is odd}}\{2^{s}q\pm 1\}).$$
The following main result shows an exact minimal decomposition of the Chebyshev polynomials. It turns out that in the case that $m\geq 3$ is an odd number, the decomposition depends only on the number $s(m)$. 
\begin{theorem}\label{decomposition}
Let $T_{m} (m\geq 2)$ be a Chebyshev polynomial defined as in  the formula {\rm (\ref{def-poly})}.\\
\indent {\rm (i)} \ If $m$ is even, then $1$ is an attracting fixed point of $T_{m}$, and all other points lie in the attracting basin of $1$.\\
\indent {\rm(ii)} \  If $m$ is odd with $s=s(m)\geq 2$ being defined in (\ref{def-s}), 
then 
     $\mathbb{Z}_{2}$ is decomposed as 
    $$\mathbb{Z}_{2}=\{0,1,-1\}\bigsqcup( E_{1}\bigsqcup  E_{2}\bigsqcup E_{3}),$$
    where \begin{align*}
           E_{1} &= \bigsqcup_{n\ge 1}\bigsqcup_{0\leq i < 2^{s-1}}E_1(n,i),\\
           E_{2} &= \bigsqcup_{n\ge 2}\bigsqcup_{0\leq i < 2^{s}}E_2(n,i),\\
           E_{3} &= \bigsqcup_{n\ge 2}\bigsqcup_{0\leq i < 2^{s}}E_3(n,i),
      \end{align*}
with
\begin{align*}
E_1(n, i)&:= 2^{n}(1+2i)+2^{n+s}\mathbb{Z}_{2}\ (n\geq 1, 0\leq i < 2^{s-1}),\\
E_2(n,i)&:= 1+2^{n}(1+2i)+2^{n+s+1}\mathbb{Z}_{2} \ (n\geq 2, 0\leq i < 2^{s}),\\
E_3(n,i)&:=-1+2^{n}(1+2i)+2^{n+s+1}\mathbb{Z}_{2} \ (n\geq 2, 0\leq i < 2^{s})
\end{align*}
being the minimal components of $T_{m}$.
\end{theorem}


The Chebyshev polynomials are important examples in arithmetic dynamical systems. See for example, in pages 29-30, 41, 328-336, and 380-381 of the book \cite{SilvermanGTM241}. 
Recently in \cite{RSWZ} the authors studied  the  Chebyshev polynomials as  dynamical system on the finite fields $\mathbb{Z}/p\mathbb{Z}$ for general prime $p\geq 2$. In \cite{DS14} the periodic orbits of Chebyshev polynomials considered as dynamics on the field $\mathbb{C}_p$ of complex $p$-adic number were studied.

 We also point out that the prime $p=2$ behaves differently from other primes, and the minimal decomposition of the dynamics of Chebyshev polynomials on $\mathbb{Z}_p$ with $p\geq 3$ needs much more difficult calculations.

\bigskip
\section{Induced dynamics of Polynomials on $\mathbb{Z}/p^n \mathbb{Z}$
}\label{top-2-adic}

In this section we will recall the techniques used in \cite{FL11} and \cite{FanSLiao} for the minimal decomposition.  The idea of these techniques was originally from \cite{DZunpu}. It was later fully developed in \cite{FL11}. 

 Let $p\geq 2$ be an arbitrary prime number. Denote by  $\mathbb{Z}_p[x]$ the set of polynomials with coefficients in $\mathbb{Z}_p$.
 Let $f \in \mathbb{Z}_p[x]$ and let  $n\ge 1$ be a positive integer.  We denote   by $f_n$ the induced
map of $f$ on $\mathbb{Z}/p^n\mathbb{Z}$, i.e.,
$$f_n (x \ {\rm mod} \ p^n)=f (x)  \quad ({\rm mod} \ p^n).$$
 Many properties of the dynamics $f$ are linked to those of
$f_n$.  For a subset $E$ of $\Zp$, denote 
$$E/p^n\Zp:=\{x\in \Zp/p^n\Zp:\  \exists \ y\in E \text{ such that } x\equiv y  \quad ({\rm mod} \ p^n)\}.$$
The following lemma can be found in \cite[p.\,111]{Ana94}, \cite[Theorem 1.2]{Ana06} and \cite[Corollary 4]{CFF09}.
\begin{lemma}[\cite{Ana94, Ana06, CFF09}]\label{minimal-part-to-whole}
Let $f\in \mathbb{Z}_p[x]$ and $E\subset \mathbb{Z}_p$ be a compact
$f$-invariant set. Then $f: E\to E$ is minimal if and only if $f_n:
E/p^n\mathbb{Z}_p \to E/p^n\mathbb{Z}_p$ is minimal for each $n\ge
1$.
\end{lemma}
It is clear that if $f_n: E/p^n\mathbb{Z}_p \to E/p^n\mathbb{Z}_p$
is minimal, then so is $$f_m: E/p^m\mathbb{Z}_p \to E/p^m\mathbb{Z}_p \text{ for each }1\le  m <n.$$ Therefore, Lemma \ref{minimal-part-to-whole} shows that to obtain the minimality of $E$,
it is important to investigate under what condition, the minimality
of $f_{n}$ implies that of $f_{n+1}$.

Let $\sigma=(x_1, \cdots, x_k) \subset
\mathbb{Z}/p^n\mathbb{Z}$ be a {\it cycle} of $f_n$ of length $k$ (also
called  a {\it $k$-cycle}), i.e.,
$$ f_n(x_1)=x_2, \cdots, f_n(x_i)=x_{i+1}, \cdots, f_n(x_k)=x_1.$$
In this case we also say $\sigma$ is {\it at level $n$}. Let
$$ X_{\sigma}:=\bigsqcup_{i=1}^k  X_i \ \ \mbox{\rm where}\ \
X_i:=\{ x_i +p^nt+p^{n+1}\mathbb{Z}; \ t=0, \cdots,p-1\} \subset
\mathbb{Z}/p^{n+1}\mathbb{Z}.$$
 Then
\[
f_{n+1}(X_i) \subset X_{i+1}  \ (1\leq i \leq k-1) \ \ \mbox{\rm
and}\ \  f_{n+1}(X_k) \subset X_1.\]

In the following we shall study the behavior of the finite dynamics
$f_{n+1}$ on the $f_{n+1}$-invariant set $X$ and determine all
the cycles in $X$ of $f_{n+1}$, which will be  called  {\it lifts} of
$\sigma$ (from level $n$ to level $n+1$). Remark that the length of a lift of
$\sigma$ is a multiple of $k$.

 Let $g:=f^k$ be the $k$-th iterate of $f$. Then any point in $\sigma$
 is fixed by $g_n$, the $n$-th induced map of $g$.
Let $$\mathbb{X}_i:=x_{i}+p^{n}\Zp =\{x\in \Zp: x\equiv x_i \ ({\rm mod} \ p^n)\}$$ be the closed disk of radius $p^{-n}$  corresponding  to $x_i\in \sigma$ and
  $\mathbb{X}_{\sigma}:=\bigsqcup_{i=1}^{k} \mathbb{X}_i$ be the clopen set corresponding to the cycle $\sigma$.
For $x\in \mathbb{X}_{\sigma}$, denote
\begin{eqnarray}
& &a_n(x):=g'(x)=\prod_{j=0}^{k-1} f'(f^j(x)) \label{def-an} \\
& &b_n(x):=\frac{g(x)-x}{p^n}=\frac{f^k(x)-x}{p^n}.\label{def-bn}
\end{eqnarray}
The values of the functions
$a_n$ and $b_n$ are important for our purpose. They define an affine
map
$$\Phi(x,t)=b_n(x)+a_n(x) t \qquad (x \in \mathbb{X}_{\sigma}, t \in \{0, \cdots,p-1\}).$$
The $1$-order Taylor expansion of $g$ at $x$ implies that for $0\leq t \leq p-1$,
\begin{eqnarray}\label{linearization}
 g(x+p^n t) \equiv x+p^n b_n(x) + p^n a_n(x) t
 \equiv x + p^n \Phi(x,t) \quad ({\rm mod} \ p^{2n}).
 \end{eqnarray}
 The function $\Phi(x,\cdot)$ is usually considered as an induced function from $\mathbb{Z}/p\mathbb{Z}$ to
 $\mathbb{Z}/p\mathbb{Z}$ by taking ${\rm mod} \ p$. We will keep the notation $\Phi(x,\cdot)$ if there is no confusion.
Then the formula (\ref{linearization}) implies that $g_{n+1}:
X_i \to X_i$ is conjugate to the linear map $$ \Phi(x_i, \cdot):
\mathbb{Z}/p\mathbb{Z} \to \mathbb{Z}/p\mathbb{Z}.
$$
Thus $ \Phi(x_i, \cdot)$ is a {\em linearization} of $g_{n+1}: X_i \to X_i$.

As proved in Lemma 1 of \cite{FL11}, the coefficient $a_n(x)$ (mod $p$) is always constant on
$\X_i$ and the coefficient $b_n(x)$ (mod $p$) is also constant on
$\X_i$ but under the condition $a_n(x)\equiv 1$ (mod $p$).
For simplicity, sometimes we write $a_n$ and $b_n$
without mentioning $x$.  

In this note, we study the dynamical systems of Chebyshev Polynomials on $\mathbb{Z}_2$.  So we assume  $p=2$  in the reminder of this section.
For the general prime $p$, we refer the  readers to  \cite{FL11,FanSLiao}.

%
%
%

From the values of $a_n$ and $b_n$, one can predict the behaviors of $f_{n+1}$ on $X_{\sigma}$, by the linearity of the map $\Phi=\Phi(x, \cdot)$:\\
 \indent {\rm (a)} If $a_n \equiv 1 \ ({\rm mod} \ 2)$ and
 $b_n \not\equiv 0 \ ({\rm mod} \ 2)$, then $\Phi$ preserves a single cycle of length $2$. So 
 $f_{n+1}$ restricted to $X_{\sigma}$ preserves a single cycle of length $2k$. In this case we say $\sigma$ {\it grows}.
Furthermore, we say $\sigma$ {\it strongly grows} if $a_n \equiv 1 \ ({\rm mod} \
4)$ and $b_n \equiv 1 \ ({\rm mod} \ 2)$,  and $\sigma$ {\it weakly
grows} if $a_n \equiv 3 \ ({\rm mod} \ 4)$ and $b_n \equiv 1 \ ({\rm
mod} \ 2)$. \\
 \indent {\rm (b)} If $a_n \equiv 1 \ ({\rm mod} \ 2)$ and
 $b_n \equiv 0 \ ({\rm mod} \ p)$, then $\Phi$ is the identity. So $f_{n+1}$ restricted to $X_{\sigma}$
 preserves
 $p$ cycles  of length $k$. In this case we say $\sigma$ {\it splits}. Furthermore, we say $\sigma$ {\it strongly splits} if $a_n \equiv 1 \
({\rm mod} \ 4)$ and $b_n \equiv 0 \ ({\rm mod} \ 2)$, and $\sigma$
{\it weakly splits} if $a_n \equiv 3 \ ({\rm mod} \ 4)$ and $b_n
\equiv 0 \ ({\rm mod} \ 2)$.\\
 \indent {\rm (c)} If $a_n \equiv 0 \ ({\rm mod} \ 2)$, then $\Phi$ is constant. So $f_{n+1}$ restricted to $X_{\sigma}$
 preserves
 one cycle of length $k$ and  the remaining points of $X_{\sigma}$ are mapped into this cycle.
  In this case we say $\sigma$ {\it grows tails}. \\
Sometimes, the behavior of cycles will be inherited. For example, we have the following lemma.
\begin{lemma}[\cite{DZunpu}, see also \cite{FL11}, Proposition 1]\label{grows-tails}
If a cycle grows tails at a certain level, then its single lift
also grows tails.
\end{lemma}

As corollary of Lemma \ref{grows-tails}, one has the following proposition.
\begin{proposition}[\cite{FL11}, p.2123]\label{attracting-orbit}
If $\sigma=(x_1, \cdots, x_k)$ is a growing tails cycle at level $n$, then $f$ has a $k$-periodic orbit in the clopen set $\mathbb{X}_{\sigma}=\bigsqcup_{i=1}^k  x_i +2^{n}\mathbb{Z}_2$ with $\mathbb{X}_{\sigma}$ as its attracting basin.
\end{proposition}

\medskip
 The case of growing tails has been checked in Proposition \ref{attracting-orbit}. For growing and splitting cases, we need do further investigations.
In fact, both of the cases of growing and splitting will be divided into two sub-cases.

The following lemma shows  that  the strong growing properties will be inherited.
\begin{lemma}[\cite{FL11}, Proposition 5]\label{grows}
Let $\sigma$ be a cycle of $f_n$ ($n\geq 2$). If $\sigma$ strongly
grows then the lift of ${\sigma}$ strongly grows.
\end{lemma}

Applying Lemma \ref{grows} consecutively, if $\tilde{\sigma}$ is the lift of $\sigma$, then $\tilde{\sigma}$ also strongly grows and again the lift of $\tilde{\sigma}$ also strongly grows and so on. In this case, we also say that the cycle $\sigma$ {\it strongly grows forever}.

Using Lemma \ref{grows}, we deduce that a strong growing cycle produces a minimal component.
\begin{proposition}\label{minimal-comp}
If $\sigma=(x_1, \cdots, x_k)$ is a strong growing cycle at level $n$, then $f$ restricted onto the invariant clopen set $\mathbb{X}_{\sigma}=\bigsqcup_{i=1}^k  x_i +p^{n}\mathbb{Z}_p$ is minimal.
\end{proposition}
\begin{proof}
By Lemma  \ref{grows}, the cycle $\sigma$ strongly grows forever. Hence, $f_m$ is a single cycle (thus minimal) on $\mathbb{X}_\sigma/ p^m \mathbb{Z}_p$ for each $m\geq n$. Therefore, by Lemma  \ref{minimal-part-to-whole}, $f$ is minimal on $\mathbb{X}_\sigma$. 
\end{proof}

\bigskip
\section{Dynamics of Chebyshev mapping  on $\mathbb{Z}_{2}$}



Let $v_2(\cdot)$ be the $2$-adic valuation on $\mathbb{Z}_2$. We will first calculate the $2$-adic valuations of the coefficients of the Chebyshev polynomials.
\begin{lemma}\label{ord2c}
Let $m=2k+1\geq 3$ be an odd number and $$T_{m}(x)=\sum_{i\geq 0}^{k}c_{2i+1}x^{2i+1}$$
be the Chebyshev polynomial of degree $m$. Let   $s=\max\{v_{2}(m+1),v_{2}(m-1)\}$. Then $v_{2}(c_{3})=s$ and $v_{2}(c_{2i+1})\geq s+1$ for $1<i\leq k$.
\end{lemma}
\begin{proof}
We distinguish two cases:\\
\indent {\rm Case 1).}  If $k$ is an even integer, then
$v_{2}(m+1)=v_{2}(2k+2)=v_{2}(2)+v_{2}(k+1)=1$ and $v_{2}(m-1)=v_{2}(2k)\geq 2$. Let $s=v_{2}(m-1)\geq 2$.
Write $m=2^{s}q+1$ for some odd integer $q$.
Then \begin{align*}
       T_{m}(x) &= \sum_{i=0}^{2^{s-1}q}(-1)^{i}\frac{m}{m-i}{ m-i \choose i }2^{m-2i-1}x^{m-2i}.\\
       \end{align*}
Making a change of variables $i=2^{s-1}q-j$, we have          
       \begin{align*}
           T_{m}(x)       &= \sum_{i=0}^{2^{s-1}q}(-1)^{i}\frac{2^{s}q+1}{ 2^{s-1}q+i+1}{2^{s-1}q+i+1 \choose 2i+1 }
                2^{2i}x^{2i+1}.
     \end{align*}
So $$c_{2i+1}=(-1)^{i}\frac{2^{s}q+1}{ 2^{s-1}q+i+1}{2^{s-1}q+i+1 \choose 2i+1}2^{2i}, \text{ for }  0\leq i \leq k.$$ Thus
$$v_{2}(c_{3})=v_{2}\left(-\frac{(2^{s}q+1)(2^{s-1}q+1)2^{s-1}q}{3!}\cdot 2^{2}\right)=s-1-1+2=s.$$
Furthermore, 
\begin{align*}
c_{2i+3}&=(-1)^{i+1}\frac{2^{s}q+1}{ 2^{s-1}q+i+2}{2^{s-1}q+i+2 \choose2i+3}2^{2i+2}\\
       &=(-1)^{i+1}\frac{(2^{s}q+1)(2^{s-1}q+i+1)!}{(2i+3)!(2^{s-1}q-i-1)!}\cdot 2^{2i+2}     \\
       &= -\frac{(2^{s-1}q+i+1)(2^{s-1}q-i)}{(2i+2)(2i+3)} \cdot 2^{2}\cdot c_{2i+1}\\
       &= -\frac{(2^{s-1}q+i+1)(2^{s-1}q-i)}{(i+1)(2i+3)} \cdot 2\cdot c_{2i+1}.
\end{align*}

If $1\leq i<\min\{2^{s}-1,k-1\}$, then
  $$v_{2}(i+1)\leq s-1$$ and  $$v_{2}(2^{s-1}q+i+1)\geq \min\{v_{2}(2^{s-1}q),v_{2}(i+1)\}=v_{2}(i+1).$$
 Observing  $v_{2}(2i+3)=0$, we obtain
  \begin{align}
  v_{2}(c_{2i+3}) &= v_{2}(c_{2i+1})+ v_{2}(2^{s-1}q+1+i)+v_{2}(2^{s-1}q-i)-v_{2}(i+1)+1 \nonumber\\
                   &\geq v_{2}(c_{2i+1})+1 . \label{vci}
\end{align}

 If $2^{s}-1\leq i\leq k$, we claim 
    $$v_{2}(c_{2i+1})\geq 2i.$$
   In fact,  since \begin{align*}
            c_{2i+1}&= (-1)^{i}\frac{2^{s}q+1}{ 2^{s-1}q+1+i}{2^{s-1}q+1+i \choose 2i+1 }2^{2i}= (-1)^{i}\frac{2^{s}q+1}{2i+1}{2^{s-1}q+i \choose 2i}2^{2i},
          \end{align*}
   we have   $$v_{2}(c_{2i+1})= v_{2}\left({2^{s-1}q+i \choose 2i}\right)+2i\geq 2i.$$
  Noting $i\geq 2^{s}-1$ and $s\geq 2$,  we get
   \begin{equation}\label{vci2}
   v_{2}(c_{2i+1})\geq 2i=2^{s+1}-2 >s+2.
   \end{equation}
By combining inequalities (\ref{vci}) and (\ref{vci2}), we have $v_{2}(c_{2i+1})\geq s+1$ for $1<i\leq k$. \\
 
\indent {\rm  Case 2).}  If $k$ is an odd number, then
$v_{2}(m+1)=v_{2}(2k+2)=v_{2}(2)+v_{2}(k+1)\geq 2$ and $v_{2}(m-1)=v_{2}(2k)=1$. Let  $s=v_{2}(m+1)\geq 2$.
Write  $m=2^{s}q-1$  with $q$ being  an odd integer.
Then by change of variables \begin{align*}
       T_{m}(x) &= \sum_{i=0}^{\lfloor m/2 \rfloor}(-1)^{i}\frac{m}{m-i}{m-i\choose i}
                                                     2^{m-2i-1}x^{m-2i} \\
                &= \sum_{i=0}^{2^{s-1}q-1}(-1)^{i+1}\frac{2^{s}q-1}{ 2^{s-1}q+i}{2^{s-1}q+i\choose 2i+1}2^{2i}x^{2i+1}.
     \end{align*}
Thus $$c_{2i+1}=(-1)^{i+1}\frac{2^{s}q-1}{ 2^{s-1}q+i}{2^{s-1}q+i\choose 2i+1 }2^{2i}, \text{ for } 0\leq i \leq k.$$
As in Csae 1),  we obtain
$$v_{2}(c_{3})=v_{2}\left(\frac{(2^{s}q-1)(2^{s-1}q)(2^{s-1}q-1)}{3!}\cdot 2^{2}\right)=s-1-1+2=s.$$
Furthermore, the relation between $c_{2i+3}$ and $c_{2i+1}$ is follows:
\begin{align*}
c_{2i+3}&=(-1)^{i+2}\frac{2^{s}q-1}{ 2^{s-1}q+i+1}{2^{s-1}q+i+1\choose 2i+3}2^{2i+2}\\
       &= (-1)^{i+2}\frac{(2^{s}q-1)(2^{s-1}q+i)!}{(2i+3)!(2^{s-1}q-i-2)!}\cdot 2^{2i+2}     \\
       &= -\frac{(2^{s-1}q+i)(2^{s-1}q-i-1)}{(2i+2)(2i+3)} \cdot 2^{2}\cdot c_{2i+1}\\
       &= -\frac{(2^{s-1}q+i)(2^{s-1}q-i-1)}{(i+1)(2i+3)} \cdot 2\cdot c_{2i+1}.
\end{align*}
 If $1\leq i<\min\{2^{s}-1,k-1\}$, then
  $$v_{2}(i+1)\leq s-1$$ and  $$v_{2}(2^{s-1}q-i-1)\geq \min\{v_{2}(2^{s-1}q),v_{2}(i+1)\}=v_{2}(i+1).$$
   Notice that  $v_{2}(2i+3)=0$.
Then
  \begin{align}
  v_{2}(c_{2i+3}) &= v_{2}(c_{2i+1})+ v_{2}(2^{s-1}q+i)+v_{2}(2^{s-1}q-i-1)-v_{2}(i+1)+1 \nonumber\\
                   &\geq v_{2}(c_{2i+1})+1. \label{vciodd}
\end{align}
Now,  we claim that
    $$v_{2}(c_{2i+1})\geq 2i, \ \text{ if }  2^{s}-1\leq i\leq k-1.$$
In fact, by observing  \begin{align*}
            c_{2i+1}= (-1)^{i}\frac{2^{s}q+1}{ 2^{s-1}q+i}{2^{s-1}q+i\choose 2i+1}2^{2i}
                    = (-1)^{i}\frac{2^{s}q+1}{2i+1}{2^{s-1}q+i-1\choose 2i }2^{2i},
          \end{align*}
   we obtain  $$v_{2}(c_{2i+1})= v_{2}\left({2^{s-1}q+i-1\choose 2i}\right)+2i\geq 2i.$$
   Since $i\geq 2^{s}-1$ and $s\geq 2$,  we have 
   \begin{equation}\label{vci2odd}
   v_{2}(c_{2i+1})\geq 2i=2^{s+1}-2 >s+2.
   \end{equation}
Therefore, by inequalities (\ref{vciodd}) and (\ref{vci2odd}), we conclude that $v_{2}(2i+1)\geq s+1$ for $1<i\leq k$. \\
\end{proof}

Now we can give the proof of our main theorem.
\begin{proof}[Proof of Theorem \ref{decomposition}]
 {\rm (1)} If $m$ is even, then $T_{m}$ induces the constant map $1$ on $\mathbb{Z}_{2}/2\mathbb{Z}_{2}$. Thus at level $1$, the point $0$ is sent to the point $1$ and $\sigma=(1)\subset \mathbb{Z}_{2}/2\mathbb{Z}_{2}$ is a fixed point. Furthermore, we can check  $$a_{1}(1)=(T_m)^{\prime}(1)\equiv0 \ ({\rm mod} \ 2).$$
Hence $\sigma$ grows tails. By Lemma \ref{grows-tails}, the single lift $\tilde{\sigma}$ of $\sigma$ grows tails. Therefore, $1$ is a fixed point, and all other points of $1+2\mathbb{Z}_2$ lie in the attracting basin of $1$. Since $T_m(2\mathbb{Z}_2)\subset 1+2\mathbb{Z}_2$,  we have $$\lim_{k\rightarrow \infty}T_{m}^k(x) =\lim_{k\rightarrow \infty}T_{m^k}(x)=1\quad \forall x\in \mathbb{Z}_{2}.$$

{\rm (2)} Assume that $m$ is odd. It is easy to check that $0,1,-1$ are fixed points of $T_{m}$. 
By Lemma \ref{ord2c}, it can be checked that $T^{\prime}_{m}(x)\equiv1 \ ({\rm mod} \ 4)$ for any $x\in \mathbb{Z}_2$. 
Thus a cycle $\sigma$ with length $1$ either strongly grows or strongly splits. Let $s=s(m)\geq 2$ be defined in (\ref{def-s}). Then $m=2^sq+1$ or $m=2^sq-1$ for some odd number $q\geq 1$. Since the treatments for both of the two cases are the same, and the minimal decompositions are also the same, we will omit the proof of the case $m=2^sq-1$ and we suppose that 
  $m=2^{s}q+1$, for some odd number $q\geq 1$. We will give the minimal decomposition on $2\mathbb{Z}_{2}$ and on $\pm 1+4\mathbb{Z}_{2}$ separately. Then the minimal decomposition of the whole space $\mathbb{Z}_{2}$ will be obtained directly. 
  
\smallskip 

  \noindent\textbf{Decomposition of} $2\mathbb{Z}_{2}$: Since $0$ is a fixed point, for $n\geq 1$, $\sigma=(0) \subset \mathbb{Z} / 2^{n} \mathbb{Z} $ is a cycle of length $1$ of $(T_{m})_n$ at level $n$. Hence $\sigma$ strongly grows or strongly splits. However, by  definition (\ref{def-bn}) of $b_n$, 
  $$\ b_{n}(0)\equiv 0 \m2{}.$$ So $\sigma$ strongly  splits. Thus $\sigma_{1}=(0)$ and $\sigma_{2}=(2^{n})\subset \mathbb{Z}/2^{n+1}\mathbb{Z}$ are two lifts of  $\sigma$ at level $n+1$. It can be checked  that $\sigma_{1}$ splits.
 Let $x_{0}= 2^{n}+t 2^{n+1}$  for some $t \in \Z2.$  Then 
  \begin{align}
  T_{m}(x_{0})-x_{0} &=(c_{1}-1)x_{0}+\sum_{i=1}^{2^{s-1}q}c_{2i+1}x_{0}^{2i+1}.
  \end{align}
 By Lemma \ref{ord2c}, we get that 
 $$v_{2}( T_{m}(x_{0})-x_{0})=n+s.$$
 By  Lemma \ref{grows}, we deduce that $\sigma_{2}$  strongly splits $s-1$ times then all lifts at level $n+s-1$ grow forever.
 By Proposition \ref{minimal-comp}, we have the following decomposition
  $$2\mathbb{Z}_{2}=\{0\}\bigsqcup \left(\bigsqcup_{n\ge 1}\left(\bigsqcup_{0\leq i < 2^{s-1}}2^{n}(1+2i)+2^{n+s}\mathbb{Z}_{2}\right)\right),$$
  with fixed point $0$ and the minimal components:
  $$2^{n}(1+2i)+2^{n+s}\mathbb{Z}_{2}\ \ (n\geq 1, 0\leq i < 2^{s-1}).$$
  
\smallskip


\noindent\textbf{Decomposition of} $\pm 1+4\mathbb{Z}_{2}$: 
 Let $x_{0}= 1+2^{n}t$ for some $n\geq 2$ and $t\in 1+ 2\mathbb{Z}_{2}$. Then 
   \begin{align*}
     T_{m}(1+2^{n}t)-(1+2^{n}t) &= \sum_{i=0}^{k}c_{2i+1}(1+2^{n}t)^{2i+1}-(1+2^{n}t) \\
     &= (1+2^{n}t)\left(c_{1}-1+\sum_{i=1}^{k}c_{2i+1}(1+2^{n}t)^{2i}\right). \\
   \end{align*}
Since $1$ is a fixed point of $T_{m}$, we have 
\begin{equation}\label{sum}
\sum\limits_{i=0}^{k}c_{2i+1}=1,
\end{equation}
and for $1\leq i \leq k$,
$$c_{2i+1}(1+2^{n}t)^{2i}-c_{2i+1}=\sum_{j=1}^{2i}c_{2i+1}{2i\choose j}2^{jn}t^{j}.
$$
Note that for  $n\geq 2$,  we have  $$v_{2}\left({2i\choose j}2^{jn}\right)>n+1 \text{ for all  } 1<j\leq 2i.$$
Thus for  $1<i\leq k$, 
\begin{equation}\label{xxxx}
v_{2}\left(c_{2i+1}(1+2^{n}t)^{2i}-c_{2i+1}\right)=v_{2}(c_{2i+1})+n+1.
\end{equation}
 By Lemma \ref{ord2c} and equalities  (\ref{sum})  and  (\ref{xxxx}), 
   \begin{equation}\label{vbn}
    v_{2}\left(T_{m}(1+2^{n}t)-(1+2^{n}t)\right)= v_{2}(2^{n+1}c_{3})=s+n+1.
   \end{equation}

For all $n\geq 2$, the cycle $(1)$ of $(T_{m})_{n}$ always splits to be two cycles $(1)$ and $(1+2^{n})$ of  $(T_{m})_{n+1}$.
Let us consider the cycle $(1+2^{n})$ of  $(T_{m})_{n+1}$.

By (\ref{vbn}) we have
$$b_{n+1}(1+2^{n})=\frac{T_{m}(1+2^{n})-(1+2^{n})}{2^{n+1}} \equiv 0 \ ({\rm mod} \ 2).$$
Thus the cycle $(1+2^{n})$ strongly splits.

The formula  (\ref{vbn}) also implies that
$$b_{n+i}(1+2^{n}+2^{n+1}h) \equiv 0 \ ({\rm mod} \ 2) \quad \text{ if }0\leq i  \leq s-1\text{ and } 0\leq h \leq 2^{i}$$ and 
$$\forall x\in 1+2^{n}\mathbb{Z}_{2}, \quad b_{n+1+s}(x)\not\equiv 0  \ ({\rm mod} \ 2).$$
Thus all the lifts of $(1+2^{n})$ split $s-1$ times then all lifts at level $n+s+1$ strongly grow. Therefore, for each  $n\geq 2$, $1+2^{n}\mathbb{Z}_{2}$ consists of  $2^{s}$ minimal components 
$$1+2^{n}(1+2i)+2^{n+s+1}\mathbb{Z}_{2} \ \ (0\leq i < 2^{s}).$$

Similarly, for each $n\geq 2$, we can decompose $-1+2^{n}\mathbb{Z}_{2}$ as $2^{s}$  minimal components $$-1+2^{n}(1+2i)+2^{n+s+1}\mathbb{Z}_{2} \ \ (0\leq i < 2^{s}).$$

We can now give the minimal decomposition of $\pm 1+4\mathbb{Z}_{2}$ as follows $$\pm1+4\mathbb{Z}_{2}=\{\pm1\}\bigsqcup\left(\bigsqcup_{n\ge 2}\left(\bigsqcup_{0\leq i < 2^{s}}\pm 1+2^{n}(1+2i)+2^{n+s+1}\mathbb{Z}_{2}\right)\right).$$
%
%
\end{proof}

\medskip

\bibliographystyle{plain}

\end{document}